\documentclass{article}
\usepackage{anysize}
\marginsize{1in}{1in}{1in}{1in}
\title{Mahler Measure of ``Almost" Reciprocal Polynomials}
\author{J.C. Saunders\footnote{email:j8saunder@uwaterloo.ca} \footnote{Research of J.C. Saunders was supported by NSERC and the Queen Elizabeth II Graduate Scholarship in Science and Technology program} \\ University of Waterloo}
\usepackage{amssymb}
\usepackage{amsmath}
\usepackage{amsthm}
\usepackage{mathtools}
\newtheorem{defn}{Definition}[section]
\newtheorem{thm}{Theorem}[section]
\newtheorem{lemma}{Lemma}[section]
\newtheorem{note}{Note}[section]

\newtheorem{ex}{Example}[section]
\newtheorem{remark}{Remark}[section]
\begin{document}
\date{}
\maketitle
\begin{abstract}
Here we give a lower bound of the Mahler measure on a set of polynomials that are ``almost" reciprocal. Here ``almost" reciprocal means that the outermost coefficients of each polynomial mirror each other in proportion, while this pattern breaks down for the innermost coefficients.
\end{abstract}
\section{Introduction}
The Mahler measure of a polynomial $f$ with integer coefficients, denoted as $M(f)$, is defined to be the absolute value of the product of all of its roots having absolute value at least $1$ and its leading coefficient. If no such roots exist, the Mahler measure is defined to be the absolute value of the leading coefficient. In other words, if
\begin{equation*}
f(x)=a_n(x-\alpha_1)(x-\alpha_2)...(x-\alpha_n),
\end{equation*}
then
\begin{equation*}
M(f)=|a_n|\prod_{i=1}^n\max\{1,|\alpha_i|\}.
\end{equation*}
A major open problem dealing with Mahler measure is whether it can get arbitrarily close to $1$ without actually being $1$. More specifically, for any $\epsilon>0$, does there exist a polynomial $f$ with integer coefficients such that $1<M(f)<1+\epsilon$. This problem was first posed by Lehmer in $1933$ and has since sparked various problems in finding Mahler measures of polynomials \cite{lehmer}. Lehmer was able to show that the polynomial
\begin{equation*}
f(x)=x^{10}+x^9-x^7-x^6-x^5-x^4-x^3+x+1
\end{equation*}
has Mahler measure $M(f)=1.1762808...$. This is the smallest Mahler measure greater than $1$ that is currently known.
\newline
\newline
An important property of polynomials with regard to calculating their Mahler measures is whether they are reciprocal or not.
\begin{defn}
Let $f(x)$ be a polynomial of degree $n$. We define the reciprocal of $f(x)$ as $f^*(x):=x^nf(1/x)$. We say that $f(x)$ is a {\em reciprocal polynomial} if $f(x)=\pm f^*(x)$.
\end{defn}
In $1971$, Smyth showed that if $f$ is an irreducible polynomial with integer coefficients that doesn't have $0$ nor $1$ as a root and is not reciprocal, then $M(f)\geq M(x^3-x-1)=1.324717...$ \cite{smyth}. In $2004$, Borwein, Hare, and Mossinghoff modified Smyth's techniques to study polynomials with $\pm 1$ coefficients and polynomials where $f(x)\neq\pm f^*(x)$ and $f(x)\equiv\pm f^*(x)\pmod m$ \cite{borwein}. Given these conditions, they prove that if $m\geq 2$, then
\begin{equation*}
M(f)\geq\frac{m+\sqrt{m^2+16}}{4}
\end{equation*}
with this bound being sharp when $m$ is even.
The larger $m$ is, the more impressive this bound becomes. Here we modify Borwein, Hare, and Mossinghoff's proof techniques on achieving the above bound to study a new class of polynomials that we define to be ``$k$-nonreciprocal" for some integer $k\geq 0$. First, a definition:
\begin{defn}
Take a polynomial in $\mathbb{Z}[x]$, say $f(x)=\sum_{i=0}^na_ix^i$. For an integer $k\geq 1$, we say that $f(x)$ is $k$-nonreciprocal if $a_na_i=a_0a_{n-i}$ for all $1\leq i\leq k-1$ with $a_na_k\neq a_0a_{n-k}$.
\end{defn}
Like with Borwein, Hare, and Mossinghoff's result, we also prove that our bound is sharp and can get arbitrarily high, depending on the set of polynomials in question. More specifically, we prove the following.
\begin{thm}\label{thm:1}
Take a polynomial in $\mathbb{Z}[x]$, say $f(x)=\sum_{i=0}^na_ix^i$.  Suppose for some $k\in\mathbb{N}$, $2k\leq n$ we have $a_na_i=a_0a_{n-i}$ for all $1\leq i\leq k-1$. Let $M(f)$ denote the Mahler measure of $f$ and $\alpha=|a_ka_n-a_0a_{n-k}|$. Then
\begin{equation*}
M(f)\geq\frac{\alpha+\sqrt{\alpha^2+4(|a_0|+|a_n|)^2|a_0a_n|}}{2(|a_0|+|a_n|)}.
\end{equation*}
\end{thm}
\begin{note}
We can see that if $f(x)\in\mathbb{Z}[x]$ is $k$-nonreciprocal for some $k\geq 0$, then $\pm(x-1)f(x)$ is also $k$-nonreciprocal. Therefore, it is enough to consider polynomials where both the leading coefficient and the constant term are both positive.
\end{note}
\begin{remark}
Borwein, Hare, and Mossinghoff noted that a Corollary to their result is that if $f$ is a nonreciprocal polynomial with all odd coefficients, then
\begin{equation*}
M(f)\geq\frac{1+\sqrt{5}}{2}=1.618...
\end{equation*}
By Theorem \ref{thm:1}, however, we may replace the condition that $f$ has all odd coefficients with the condition that for the smallest $k$ that we have $a_ka_n\neq a_0a_{n-k}$, then $|a_ka_n-a_0a_{n-k}|\geq 2$. Assuming that $|a_n|=|a_0|=1$ (for otherwise $M(f)\geq\min\{|a_0|,|a_n|\}\geq 2$), this condition is substantially weaker than the condition that $f$ is nonreciprocal and has all odd coefficients.
\end{remark}
\section{Proof and Example}
Our proof follows that of Borwein, Hare, and Mossinghoff in \cite{borwein}. Unlike Borwein, Hare, and Mossinghoff, however, we allow the innermost coefficients to not necessarily adhere to the reciprocal strucutre. We use the following result by Wiener, found on pg. 392 of \cite{schinzel}.
\begin{lemma}[Wiener]
\label{lem:1}
Suppose that $\phi(z)=\sum_{i\geq 0}\gamma_iz^i$, with $\gamma_i\in\mathbb{C}$ is analytic in an open disk containing $|z|\leq 1$ and satisfies $|\phi(z)|\leq 1$ on $|z|=1$. Then $|\gamma_i|\leq 1-|\gamma_0|^2$ for $i\geq 1$.
\end{lemma}
We now prove Theorem \ref{thm:1}.
\begin{proof}[Proof of Theorem \ref{thm:1}]
Let $f(z)=\sum_{i=0}^na_iz^i=a_n(z-\alpha_1)...(z-\alpha_n)$ satisfy the hypothesis in the theorem with $a_0$ and $a_n$ both being positive. Write $f^*(z)=\sum_{i=0}^nd_iz^i$ so that $a_0d_i=a_na_i$ for all $1\leq i\leq k-1$. Let the power series of $1/f^*(z)$ be $\sum_{i\geq 0}e_iz^i$. Then we have $e_0=1/a_n$. Let
\begin{equation*}
G(z)=f(z)/f^*(z)=\sum_{i\geq 0}q_iz^i.
\end{equation*}
It doesn't matter if $q_i\in\mathbb{Z}$ for all $i\in\mathbb{N}\cup\{0\}$ or not. We have $q_0=\frac{a_0}{a_n}$. From $f^*(z)G(z)=f(z)$, we obtain $\sum_{i=0}^jd_iq_{j-i}=a_j$. Thus, for $j\geq 1$, we have
\begin{equation*}
a_nq_j=(a_j-q_0d_j)-\sum_{i=1}^{j-1}d_iq_{j-i}.
\end{equation*}
From $a_0d_i=a_na_i$, we can see by induction that $q_i=0$ for all $1\leq i\leq k-1$ and $q_k=\frac{a_k}{a_n}-\frac{a_oa_{n-k}}{a_n^2}\neq 0$.
\newline
\newline
Let $\epsilon=-1$ if $f(z)$ has a zero of odd mulitiplicity at $z=1$ and $\epsilon=1$ otherwise. Since
\begin{equation*}
\prod_{|\alpha_i|=1}\frac{z-\alpha_i}{1-\overline{\alpha_i}z}=\prod_{|\alpha_i|=1}\frac{-\alpha_i(1-z/\alpha_i)}{1-z/\alpha_i}=\prod_{|\alpha_i|=1}(-\alpha_i)=\epsilon,
\end{equation*}
we define
\begin{equation*}
g(z):=\epsilon\prod_{|\alpha_i|<1}\frac{z-\alpha_i}{1-\overline{\alpha_i}z}
\end{equation*}
and
\begin{equation*}
h(z):=\prod_{|\alpha_i|>1}\frac{1-\overline{\alpha_i}z}{z-\alpha_i}.
\end{equation*}
so that
\begin{equation*}
\frac{g(z)}{h(z)}=\frac{\prod_{i=1}^n(z-\alpha_i)}{\prod_{i=1}^n(1-\overline{\alpha_i}z)}=\frac{\prod_{i=1}^n(z-\alpha_i)}{\prod_{i=1}^n(1-\alpha_iz)}=\frac{f(z)}{f^*(z)}=G(z).
\end{equation*}
Since all poles of both $g(z)$ and $h(z)$ lie outside the unit disk, both functions are analytic in a region including $|z|\leq 1$. Also, if $|z|=1$ and $\beta\in\mathbb{C}$, then
\begin{equation*}
\left(\frac{z-\beta}{1-\overline{\beta}z}\right)\overline{\left(\frac{z-\beta}{1-\overline{\beta}z}\right)}=\left(\frac{z-\beta}{1-\overline{\beta}z}\right)\left(\frac{1/z-\overline{\beta}}{1-\beta/z}\right)=1
\end{equation*}
so $|g(z)|=|h(z)|=1$ on $|z|=1$. Let
\begin{equation*}
g(z)=\sum_{i\geq 0}b_iz^i
\end{equation*}
and
\begin{equation*}
h(z)=\sum_{i\geq 0}c_iz^i.
\end{equation*}
Since $g(z)=h(z)G(z)$, we have $b_i=c_iq_0$ for $0\leq i<k$ and $b_k=c_0q_k+c_kq_0$. Thus
\begin{equation*}
\left|c_0\left(\frac{a_k}{a_n}-\frac{a_0a_{n-k}}{a_n^2}\right)\right|=|c_0q_k|=|b_k-c_kq_0|\leq|b_k|+|c_k|q_0.
\end{equation*}
Notice that
\begin{equation}
c_0=|h(0)|=\prod_{|\alpha_i|>1}1/|\alpha_i|=|a_n|/M(f),\label{eqn1}
\end{equation}
so that
\begin{equation}
\left|\frac{1}{M(f)}\left(a_k-\frac{a_0a_{n-k}}{a_n}\right)\right|=|c_0q_k|\leq|b_k|+|c_k|q_0.
\label{eqn2}
\end{equation}
We now divide into two cases.
By Lemma \ref{lem:1}, we have $|c_k|\leq 1-c_o^2$ and $|b_k|\leq 1-b_0^2$. Notice that $b_0=c_0q_0$. Combining \eqref{eqn1} and \eqref{eqn2}, we have
\begin{align*}
\frac{\left|a_k-\frac{a_0a_{n-k}}{a_n}\right|}{M(f)}&\leq(1-b_0^2)+(1-c_0^2)q_0\\
&=(1-c_0^2q_0^2)+(1-c_0^2)q_0\\
&=(1+q_0)(1-q_0c_0^2)\\
&=(q_0+1)\left(1-\frac{q_0a_n^2}{M(f)^2}\right)\\
&=(q_0+1)\left(1-\frac{a_0a_n}{M(f)^2}\right).
\end{align*}
Thus we have
\begin{equation*}
M(f)\left|a_k-\frac{a_0a_{n-k}}{a_n}\right|\leq(q_0+1)(M(f)^2-a_0a_n).
\end{equation*}
This gives
\begin{equation*}
M(f)\geq\frac{\left|a_k-\frac{a_0a_{n-k}}{a_n}\right|+\sqrt{\left|a_k-\frac{a_0a_{n-k}}{a_n}\right|^2+4(q_0+1)^2a_0a_n}}{2(q_0+1)}.
\end{equation*}
The result follows.
\end{proof}
\begin{note}
If $|a_na_k-a_0a_{n-k}|>|a_0^2-a_n^2|$, then the above bound is non-trivial since then it will be greater than
\begin{align*}
\frac{|a_0^2-a_n^2|+\sqrt{|a_0^2-a_n^2|^2+4(a_0+a_n)^2|a_0a_n|}}{2(a_0+a_n)}&=\frac{|a_0-a_n|+\sqrt{(a_0-a_n)^2+4a_0a_n}}{2}\\
&=\frac{|a_0-a_n|+\sqrt{(a_0+a_n)^2}}{2}\\
&=\frac{|a_0-a_n|+a_0+a_n}{2}\\
&=\max\{a_n,a_0\},
\end{align*}
which is the trivial bound.
\end{note}
\begin{note}
If $f(x)$ is a reciprocal polynomial, the above bound is trivial for we would have $a_n=a_0$ and $a_k=a_{n-k}$ so that
\begin{equation*}
|a_na_k-a_0a_{n-k}|=0.
\end{equation*}
Thus, by Theorem \ref{thm:1}, we have
\begin{equation*}
M(f)\geq\frac{\sqrt{4(a_0+a_n)^2a_ na_0}}{2(a_0+a_n)}=a_n,
\end{equation*}
which is trivial.
\end{note}
We now show some examples, indicating that our bound in Theorem \ref{thm:1} is sharp.
\begin{ex}
Let $k,n\in\mathbb{N}$ where $n>2k$ and $n\neq 3k$ and $a,b,c\in\mathbb{Z}$ such that $a>0>c$, and $a-|b|\leq -c\leq a+|b|$. Consider the polynomial $f(x)=(ax^{2k}+bx^k+c)(x^{n-2k}-1)$, which satisfies $a_na_i=a_0a_{n-i}$ for all $1\leq i\leq k-1$ and $a_na_k\neq a_0a_{n-k}$ where $f(x)=\sum_{i=0}^{n}a_ix^i$. Let $\alpha=|a_ka_n-a_0a_{n-k}|$ We have
\begin{equation*}
M(f)=\frac{\alpha+\sqrt{\alpha^2+4(a_0+a_n)^2a_0a_n}}{2(a_0+a_n)}.
\end{equation*}
Let $k,n\in\mathbb{N}$ where $n\geq 2k$ and $a,b,c\in\mathbb{Z}$ satisfying the given conditions. We have
\begin{equation*}
f(x)=ax^n+bx^{n-k}+cx^{n-2k}-ax^{2k}-bx^k-c
\end{equation*}
if $n>4k$,
\begin{equation*}
f(x)=ax^{4k}+bx^{3k}+(c-a)x^{2k}-bx^k-c,
\end{equation*}
if $n=4k$,
\begin{equation*}
f(x)=ax^n+bx^{n-k}-ax^{2k}+cx^{n-2k}-bx^k-c
\end{equation*}
if $4k>n>3k$, and
\begin{equation*}
f(x)=ax^n-ax^{2k}+bx^{n-k}-bx^k+cx^{n-2k}-c
\end{equation*}
if $3k>n>2k$.
In all cases, we can easily see that if we write $f(x)=\sum_{i=0}^{n}a_ix^n$, then we have $a_n=a$, $a_0=-c$, $a_k=-b$, $a_{n-k}=b$, $a_na_i=a_0a_{n-i}$ for all $1\leq i\leq k-1$ and $a_na_k\neq a_0a_{n-k}$. We therefore have that $\alpha=|a_na_k-a_0a_{n-k}|=|b(a-c)|$.
\newline
\newline
Since all the roots of $x^{n-2k}-1$ have absolute value $1$, we have $M(f)=M(ax^{2k}+bx^k+c)$. By the quadratic formula, the roots of $ax^{2k}+bx^k+c$ are the $k$th roots of the numbers
\begin{equation*}
\frac{-b\pm\sqrt{b^2-4ac}}{2a}.
\end{equation*}
Since $c<0<a$, the absolute values of these numbers are
\begin{equation*}
\frac{\pm|b|+\sqrt{b^2-4ac}}{2a}
\end{equation*}
First, consider
\begin{equation}
\frac{|b|+\sqrt{b^2-4ac}}{2a}\label{root1}
\end{equation}
If $|b|>a$, then clearly this number is greater than $1$ so we may assume that $|b|\leq a$. By our assumption that $-c\geq a-|b|$, we then have
\begin{align*}
\frac{|b|+\sqrt{b^2-4ac}}{2a}&\geq\frac{|b|+\sqrt{b^2+4a(a-|b|)}}{2a}\\
&=\frac{|b|+\sqrt{4a^2-4a|b|+b^2}}{2|a|}\\
&=\frac{|b|+2a-|b|}{2a}\\
&=1.
\end{align*}
Now consider
\begin{equation}
\frac{\sqrt{b^2-4ac}-|b|}{2a}\label{root2}
\end{equation}
By our assumption that $-c\leq a+|b|$, we have
\begin{align*}
\frac{\sqrt{b^2-4ac}-|b|}{2a}&\leq\frac{\sqrt{b^2+4a(a+|b|)}-|b|}{2a}\\
&=\frac{\sqrt{4a^2+4a|b|+b^2}-|b|}{2a}\\
&=\frac{2a+|b|-|b|}{2a}\\
&=1.
\end{align*}
All of the $n$th roots of \eqref{root1} have absolute value at least $1$, while all of the $n$th roots of \eqref{root2} have absolute value at most $1$. Hence
\begin{equation*}
M(f)=\frac{|b|+\sqrt{b^2-4ac}}{2}.
\end{equation*}
Note that
\begin{align*}
\frac{|b|+\sqrt{b^2-4ac}}{2}&=\frac{|ba-cb|+\sqrt{(ba-cb)^2-4(a-c)^2ca}}{2(a-c)}
\end{align*}
since $a$ and $c$ have opposite signs.
Thus we obtain our bound.
\end{ex}
\begin{note}
If we impose the restriction $a_0,a_n=\pm 1$ on this example, then we will have $\alpha$ being even. It is unknown whether the inequality in Theorem \ref{thm:1} is still sharp if we impose $a_0,a_n=\pm 1 $ and $\alpha$ being odd.
\end{note}
\section{Future Work}
In calculating the Mahler measure of reciprocal or ``almost" reciprocal polynomials, there are a few questions worth pursuing. For instance, what if we have a polynomial $f(x)$ that is ``almost reciprocal" as defined in this paper, as well as being ``almost reciprocal" as defined by Borwein Hare, and Mossinghoff in \cite{borwein}? Can we get better bounds for these polynomials than the bounds we have shown here and the bounds proved by Borwein, Hare, and Mossinghoff? Another question is if we can use the ideas presented here on sparse polynomials.
\section{Acknowledgements}
The author would like to thank Dr. Kevin Hare and Dr. Yu-Ru Liu for their support and suggestions for this paper.

\end{document}